\documentclass{amsart}

\usepackage{amsmath,amsfonts,amssymb}
\usepackage{graphicx,psfrag}

\newtheorem{thm}{Theorem}
\newtheorem{rk}{Remark}
\newtheorem{prop}{Proposition}

\newtheorem{ex}{Example}
\newtheorem{lemma}{Lemma}
\newtheorem{defi}{Definition}
\newcommand{\pf}{{\flushleft{\bf Proof: }}}
\newcommand{\R}{{\mathbb{R}}}

\newcommand{\Z}{{\mathbb{Z}}}

\newcommand{\T}{{\mathbb{T}}}

\begin{document}
\title {Branched coverings of the sphere having
a completely invariant continuum with infinitely many Wada Lakes.}
\author{J.Iglesias, A.Portela, A.Rovella and J.Xavier}

\begin{abstract}
We construct a family of smooth branched coverings of degree $2$ of the sphere $S^2$ having a
completely invariant indecomposable continuum $K$ and infinitely many Wada Lakes.
\end{abstract}

\maketitle

\section{Introduction}
This note deals with branched coverings of the sphere of degree $d$, with $|d|>1$; an example will be given having an {\em indecomposable} completely invariant continuum. There are
 two main sources of interest:\\

1.  Whether or not there exists a rational function with an
indecomposable continuum as its Julia set is a well-known unsolved
problem (see \cite{cmr}).  It is known that the pseudo-circle is its own $d$-fold covering space ($|d|>1$) (see \cite{heath} and \cite{gam}). It follows that there exist self maps of the 
annulus of degree $d$, $|d|>1$,  with an indecomposable continuum carrying a $d$ to $1$ dynamics. Of course this example can be extended to $S^2$, so we are not presenting
the first example of a degree $d$,  $|d|>1$,  map of the sphere with a $d:1$ invariant indecomposable continuum.  However, to our knowledge this is the first example of a 
{\it branched covering} of the sphere of degree $d$, with $|d|>1$, having an indecomposable completely invariant continuum. Furthermore, as will
be explained below, the nature of our example is essentially different.\\

2. The following stands as an open problem:  let $f: S^2 \to S^2$ be a continuous map of degree $d$, $|d|>1$, and let $N_nf$ denote the number of fixed points of $f^n$. When does
the growth rate inequality $\limsup \frac{1}{n} \log N_nf\geq \log |d|$ hold for $f$? (This is Problem 3 posed in \cite{shub2}). There are many
recent papers written on the subject. For example the following result holds: Assume that $f$ is a degree $d$ ($|d|>1$) branched covering of the sphere with a completely invariant
simply connected region $R$ whose boundary component is locally connected. If the set of critical points in the boundary of $R$ is not reduced to a fixed point with multiplicity $d-1$,
then $f$ satisfies the growth
rate inequality. This result appears in \cite{iprx4}, where many other references on this problem can be found. The hypothesis of local connectivity does not seem to be relevant here,
and it is plausible that this hypothesis should be replaced by decomposability instead.  It is then natural to try and understand the examples presenting indecomposable totally invariant
continua.\\

We prove the following:

\begin{thm}
\label{t1}
There exists a branched covering $f:S^2\to S^2$ supporting a completely invariant continuum $K$ and satisfying the following
properties:
\begin{enumerate}
\item
$f$ is smooth.
\item
$f$ has two attracting fixed points, each basin contains one of the two critical points of $f$. The restriction of $f$ to the immediate basins is injective.
\item
The compact set $K$ is the complementary set of the union of the basins: $K$ is a repellor with local product structure.
\item
$K$ is indecomposable, with infinitely many Wada lakes.
\item
$f$ satisfies the growth rate inequality.
\item
Every $C^1$ perturbation of $f$ satisfies properties $2, 4$ and $5$ above. Every branched covering which is also a small $c^\infty$ perturbation of satisfies all the properties above.\\

\end{enumerate}
\end{thm}

\begin{rk}

The example obtained by extending Gammon's $d$-fold covering of the pseudo-circle in \cite{gam} is of a very different nature: the indecomposable continuum is an essential pseudo-
circle contained in an annulus (thus extendable to a self map of the sphere with just two complementary components).  It is not a covering map and it is not $C^1$. It satisfies the growth rate inequality
as it follows from \cite{boronski} (see also \cite{b2} for a stronger result in this direction).\\\end{rk}

The construction of our example begins with a nonexpanding Anosov endomorphism of the torus; after performing the derived from Anosov perturbation to transform both fixed points into
attractors, the new map is carried to the sphere under a four points ramification branched covering. Then some arranges are done to obtain smoothness and the remaining properties.\\

There is another possible approach: the linear Anosov can be first carried to the sphere under the branched covering; then, imitating the construction for diffeomorphism of the Plykin
attractor (in our case, it would be a repellor), and a derived from pseudo-Anosov perturbation, the same example is obtained (see \cite{bm} where the derived from pseudo-Anosov is
performed).\\

The derived from Anosov endomorphism was first introduced by F.Przytycki, see section 6 in \cite{prz}.\\

In the next section the properties of indecomposable continua are analyzed and a criterium for a laminated continuum to be indecomposable is presented. Section 3 contains the construction
of our example and the last section explores general properties of other examples that are quotients of torus endomorphisms.

\section{Indecomposable Continua.}
In this section we present the definition and properties of indecomposable continua. For details see section 3.8 of \cite{hy}.

\begin{defi}
A continuum is a compact connected set.
A continuum is decomposable if it is the union of two proper subcontinua. Otherwise, it is indecomposable.
\end{defi}

A Wada lake of an indecomposable continuum $K$ contained in the sphere is associated to a prime end whose principal set is equal to the whole $K$.
For our purposes, a Wada lake is just a component of the complement of an indecomposable continuum $K$ which is dense in $K$.\\

\noindent
If $K$ is a proper subcontinuum of $X$ and has nonempty interior, then $X$ is decomposable, because it may happen that $X\setminus K$ is connected, in which case the pair $\{K,\overline{X\setminus K}\}$ is a decomposition of $X$ into two proper subcontinua, or it may happen it is not connected; and in this case, if $\{A,B\}$ is an open partition of $X\setminus K$, then $K\cup A$ and $K\cup B$ are connected and closed in $X$ so $X$ is decomposable. This proves one implication of the following proposition, the other is tivial.

\begin{prop}
\label{p1}
A continuum is decomposable if and only if it contains a proper subcontinuum with nonempty interior.
\end{prop}

\begin{ex}
\label{e1}
Theorem \ref{indec} below implies that the solenoid, the attractor of a derived from Anosov homeomorphism of the torus or the Plykin attractor are examples of indecomposable continua. With minor modification also the stable manifold of a horseshoe is seen to be indecomposable.
\end{ex}

It is assumed from now on in this section that $X$ is a continuum and every $x\in X$ has a neighborhood $U$ which is homeomorphic to the product of the usual Cantor set $C$ and an
open interval $I$. To avoid unnecessary notations we will drop the homeomorphism between the neighborhood and $C\times I$.\\

The {\em local leaf} of a point $\bar x=(x,a)$ relative to its neighborhood $C\times I$ is defined as the set $\{x\}\times I$. A {\em segment} of $X$ is a connected union of a finite number of local leaves.
The {\em leaf at} $x$, denoted $H_x$, is the union of the segments containing $x$.  Two leaves are disjoint or equal, because if $z$ belongs to $H_p\cap H_q$ then the local leaf of
$z$ is contained in both $H_p$ and $H_q$.\\

Let $p\in X$; define a local parametrization of $H_p$ as a pair $(\varphi,J)$ where $J$ is an finite open interval in the real line, and $\varphi$ an immersion from $J$ into $X$ whose image is contained in $H_p$.
Note that if $(\varphi_0,J_0)$ and $(\varphi_1,J_1)$ are local parametrizations of $H_p$ and $\varphi_0(J_0)\cap \varphi_1(J_1)\neq \emptyset$, then there exists a local parametrization $(\varphi,J)$ extending $(\varphi_0,J_0)$ (meaning $J_0\subset J$ and $\varphi|_{J_{0}}=\varphi_0$).
An argument of maximality implies that there exists an immersion onto $\varphi:\R\to H_p$. \\

The following assertion is proved using the definition of leaf and the local structure. It will referred below as ``continuity of leaves".
A section at $p$ is a set $\Sigma_p$ containing $p$ and having the form $C\times \{a\}$, where $U=C\times I$ is a neighborhood of $p$.
Assume now that $q\in H_p$ and that $\Sigma_p$ and $\Sigma_q$ are corresponding sections; let also $L$ be a continuum contained in $H_p$
containing $p$ and $q$, and let $W$ be a neighborhood of $L$. Then there exists a neighborhood $V$ of $p$ in $\Sigma_p$ such that, for every $y\in V$ the leaf $H_y$ contains a continuum $L_y$ contained in $W$ and intersecting $\Sigma_q$.\\

This section finishes with the proof of the following theorem, that gives a sufficient condition for a continuum to be indecomposable.

\begin{thm}
\label{indec}
If $X$ is locally the product of a Cantor set times an interval, and there exist a dense leaf, then $X$ is indecomposable.
\end{thm}

\pf
Assume that $H_z$ is dense. \\
{\em Step 1.} For every $x$ in a residual subset $X'$ of $X$ the leaf $H_x$ is dense.\\
The proof is standard: for each nonempty open set $U\subset X$ define the set $X(U)=\{x\in X: H_x\cap U\neq\emptyset\}$. By continuity of leaves, the set $X(U)$ is open; it is
also dense because $H_z$ is dense. If $\{U_n\}_{n\in \Z}$ is a countable basis of open sets, then by Baire Theorem $X':=\cap_n X(U_n)$ is a residual set, and for every $x\in X'$ the
leaf $H_x$ is dense.\\
\begin{figure}[ht]
\begin{center}
\caption{\label{phia2}}
\psfrag{sigmap}{$\Sigma_{p}$}
\psfrag{sigmaq}{$\Sigma_{q}$}
\psfrag{v}{$V^{'}$}
\psfrag{p}{$\varphi_x(t)$}
\psfrag{z}{$\varphi_x(t^{'})=z$}
\psfrag{u}{$U$}
\psfrag{w}{$W$}
\psfrag{ii}{$+\infty$}
\psfrag{x}{$x$}
\includegraphics[scale=.2]{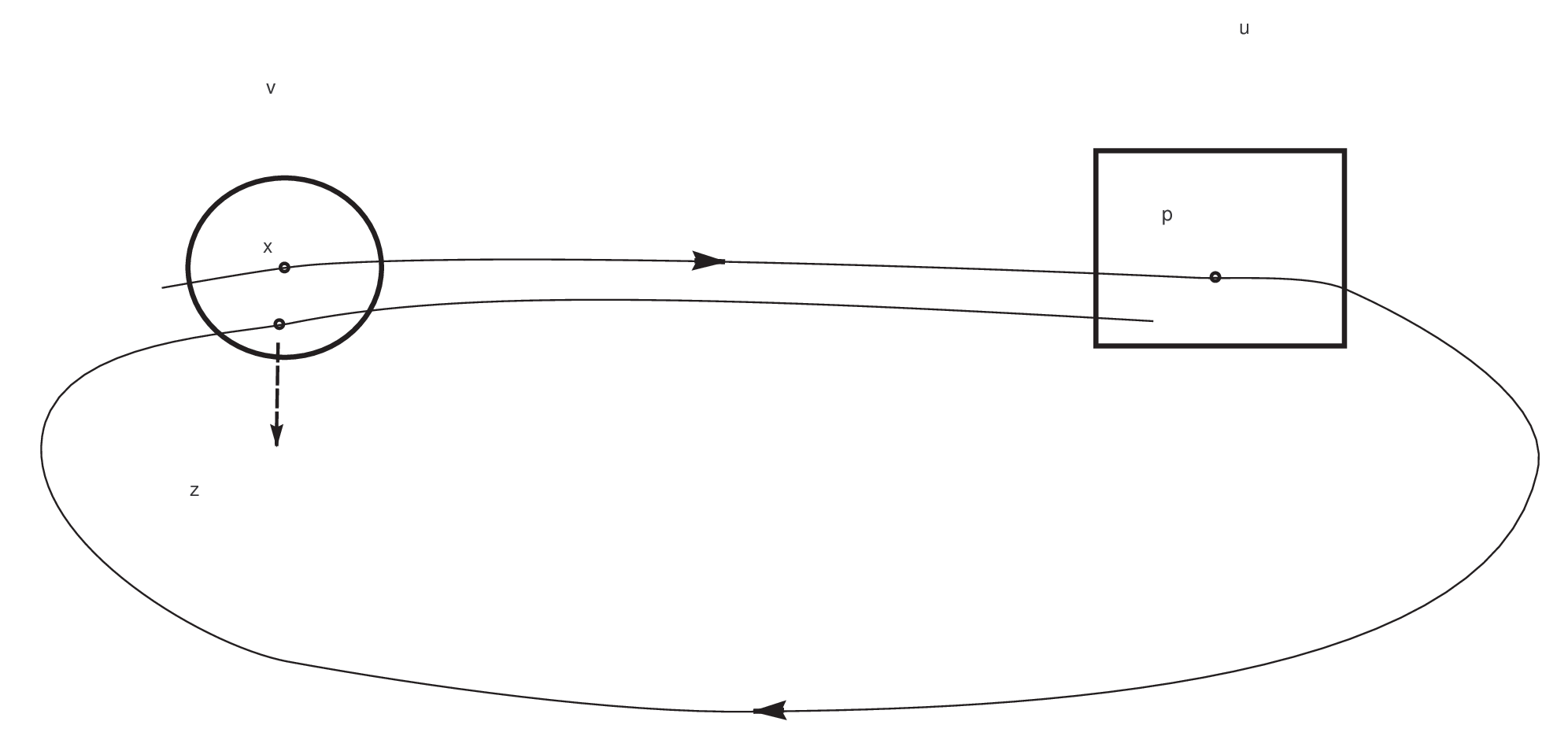}
\end{center}
\end{figure}
{\em Step 2.} For each nonempty open set $U$, define $X_0(U)$ as the set of points $x$ such that both $\varphi_x(\R^+)$ and $\varphi_x(\R^-)$ intersect $U$. It is claimed that $X_0(U)$ is dense in $X$ for every $U$.\\
Let $V$ be an open set, $x\in X'\cap V$ and $\varphi_x:\R\to H_x$ an immersion. As $H_x$ is dense, assume without loss of generality that there exists $t>0$ such that $\varphi_x(t)\in U$. Taking a smaller $V'\subset V$ if necessary, it follows from continuity of leaves that for every $y\in V'$ there exists a segment contained in $H_y$ and close to $\varphi_x([0,t])$ whose intersection with $U$ is nonempty. it follows that if $\varphi_x(T)$ intersects $V'$ for $T$ arbitrary large negative, then $x\in X_0(U)$ and the assertion follows. Otherwise, as $H_x$ is dense, there must exist $t'>t$ such that $\varphi_x(t')\in V'$. Then it is immediate that $z=\varphi_x(t')$ belongs to $X_0(U)\cap V'$. See figure 1.\\

To prove the Theorem, let $K$ be a compact proper set with nonempy interior. It will be shown that $K$ is not connected implying that $X$ is indecomposable by virtue of proposition \ref{p1}.

As $K$ has interior, then there is some nonempty open set $V\subset K$. As $K$ is not dense there exists some nonempty open set $U=C\times I$ with closure disjoint from $K$. Take $z\in V\cap X_0(U)$ as given in step 2 and a maximal segment $\ell_z$ containing $z$ and contained in $H_z\setminus U$. Choose a point $w\in K\setminus \ell_z$, which is possible because $K$ has nonempty interior. Let $p$ and $q$ be the extreme points of $\ell_z$, both belong to $C\times \{0\}$ or $C\times \{1\}$. Recall that if $A=(x,t)\in C\times I$ then $\Sigma_A=C\times \{t\}$.
As $p$ and $q$ can be seen as contained in $\Sigma_p$ (resp. $\Sigma_q$) there exist closed and open neighborhoods $C_p\subset \Sigma_p$ of $p$ and $C_q\subset \Sigma_q\cap C$ of $q$.
One can assume that these neighborhoods are small enough such that: \\
1. For every point $y\in C_p$ there exists a maximal segment $\ell_y$ close to $\ell_z$ and contained in $H_y\setminus(C\times I)$. \\
2. The point $w\in K\setminus \ell_z$ is not contained in any of the $\ell_y$. See figure 2.

\begin{figure}[ht]
\begin{center}
\caption{\label{phia2}}
\psfrag{sigmap}{$\Sigma_{p}$}
\psfrag{sigmaq}{$\Sigma_{q}$}
\psfrag{v}{$V$}
\psfrag{p}{$p$}
\psfrag{q}{$q$}
\psfrag{u}{$U$}
\psfrag{w}{$W$}
\psfrag{ii}{$+\infty$}
\psfrag{ww}{$w\in K\setminus W$}
\includegraphics[scale=.25]{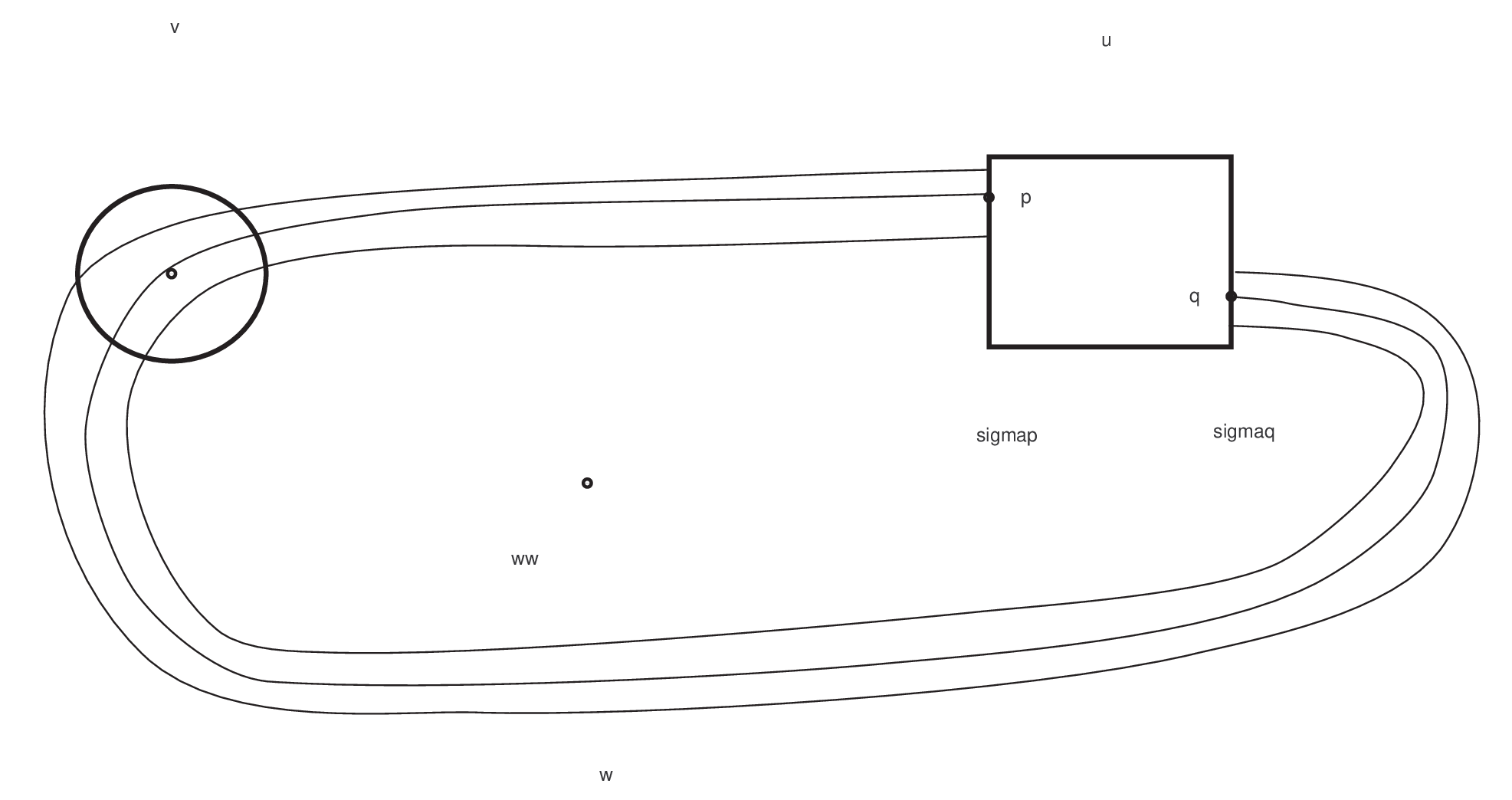}
\end{center}
\end{figure}

Consider the set $W=\cup \{\ell_{y} : y\in C_p\}$. $W$ is closed by continuity of leaves and open in the complement of $\overline U$, by the choice of $C_p$.
Now using that $W$ is closed and open in the complement of $\overline U$, that $K\cap \overline U= \emptyset$ and that $z\in K\cap W$ and $w\in K\setminus W$, it follows that $K$ is not connected, as wished.

\begin{ex}
\label{e2}
Let $\Sigma=2^\Z$ and $\sigma$ the shift. If $X$ is the suspension of $\sigma$ then it is locally Cantor times interval, there are dense leaves and closed leaves. So $X$ is indecomposable. Question: is it true that if $X$ is contained in a surface then the existence of dense leaves implies that every leaf is dense?\\
Let $h$ be a homeomorphism of the interval $[0,1]$ and $C$ the middle thirds Cantor set. Assume that $h$ carries $C\cap [0,1/9]$ to $C\cap [0,1/3]$ and $C\cap[2/9,1]$ in $C\cap [2/3,1]$. If $X$ is the suspension of the restriction of $h$ to $C$ then $X$ is a continuum which is locally the product $C\times I$ and no leaf is dense. $X$ is decomposable.
\end{ex}

\section{An explicit example}
%
% As it was already pointed out in the last section, there are various examples of indecomposable continua supporting homeomorphisms.
% Immediate attempts to construct nonivertible coverings on indecomposable continua fail: the cartesian product of an indecomposable continuum times a connected set is not
% indecomposable; if in (1) of example \ref{e2} one takes $\sigma$ as the one sided shift on $2^\N$, then its suspension is not indecomposable.

In this section we give an example of a $2:1$ branched covering map of the sphere with a $2:1$ invariant indecomposable continuum. We choose a particular example, but for every nonexpanding hyperbolic noninvertible linear endomorphism of the torus the same results hold.\\

We consider the linear hyperbolic toral endomorphism $\overline A$ induced by the matrix $$A=
  \left[ {\begin{array}{cc}
   4 & 1 \\
   2 & 1 \\
  \end{array} } \right]
$$

The map $\overline A: \T^2\to \T^2$  is a nonexpanding Anosov endomorphism of degree $2$; its dynamics is well known, it has invariant stable and unstable foliations with dense leaves.
It is {\em special} meaning that the unstable space does not depend on the preorbit chosen to define it but only on the point. This is well known to be not structurally stable: generic $C^1$ perturbations do not satisfy this property (see \cite{prz}).

Consider the equivalence relation $x \sim -x$, $x\in \T^2=\R^ 2/\Z^2$. $T^2/\sim$ is the product $[0,1/2]\times[0,1]$ with the following identifications:\\
$(x,0)\sim (x,1)$ for every $x$; $(0,y)\sim (0,1-y)$ and $(1/2,y)\sim (1/2,1-y)$ for every $y$.\\
Note that $\T^2/\sim$ is topologically a sphere and that the quotient map $\pi: \T^2\to S^2$ is a degree $2$ branched covering with exactly four critical points; namely, the classes in $\T^2$ of the points
$(1/2, 1/2),(1,0),(1/2,0),(0,1/2)$.\\

Note that $\overline A$ induces a map on the sphere simply because it is linear, but we will previously proceed to perform a derived from Anosov modification. The fixed points of
$\overline A$ are $(0,0)$ and $(1/2,1/2)$. Both can be made attractors, creating as usual two new saddles. This proceeding is well known and taking an extra care note that the
modification can made symmetric with respect to the strong stable manifold of the fixed points. It then follows that the new map $A'$ obtained still preserves the equivalent relation
$\sim$ in $T^2$. As it also well known, the unstable foliation remains the same.  It follows that $A'$ induces a map on the sphere $f'$ so that $\pi A'=f'\pi$. \\

By construction $f'$ is a degree two branched covering map; it has two critical points that can be easily found by computing those points having just one preimage. The critical points are $\pi(1/4,0),\pi(1/4,1/2)$ with corresponding critical values $\pi(0,1/2), \pi(1/2,0)$.
Moreover, the images of the critical values are the fixed points of $f'$ : $f'(\pi(0,1/2)) = \pi(1/2,1/2)$
and $f'(\pi(1/2,0)) = \pi(0,0)$.\\

Finally note that as $A$ is smooth, $f'$ is also smooth excepting exactly at the critical points and critical values where it is not even differentiable. These points belong to the basin of attraction of the fixed points. We will show how to make the critical point $c$ smooth.
In our linear case it can be done easily. Note that the preimage of an adequate ellipse $E$
around $f'(c)$ is a circle $S_\rho(c)$ of radius $\rho$ centered at $c$ such that $f'$ is two-to one from $S_r(c)$ to $E$.
Moreover, lines through $c$ are carried to lines through $f(c)$, which implies that, in polar coordinates (centered at $c$ and $f'(c)$),
one has $f'(r,\theta)=(g(r,\theta),h(\theta))$.
In addition, note that by linearity again, $\frac{g(r,\theta)}{r}$ is bounded and bounded away from $0$.
Let now $\Phi$ be a $C^\infty$ function that is equal to the identity close to $S_\rho(0)$ and is equal to $(r,\theta)\to (e^{-1/r^2},\theta)$ close to $0$ and then $f'\circ\Phi$ is
$C^\infty$ at $c$ and coincides with $f'$ outside a neighborhood of $c$.\\

Next we make a new modification of this map so that the obtained map is also smooth at the critical values (much easier because $f'$ is locally a homeomorphism at these points).
Denote by $f$ this new modified map.\\

\noindent
{\bf Proof of Theorem \ref{t1}.}
(1) It was already established that $f$ is a $C^\infty$ branched covering. \\
(2) Note that the second iterate of the critical points of $f$ are fixed points, so the postcritical set is finite and $f$ is a Thurston map. To  prove the injectivity assertion,
denote by $B_0$ the immediate basin of $(0,0)$. Observe that the fixed point $(0,0)$ of the linear map $A$ has a preimage in the torus, namely the point $(1/2,0)$, which does not
belong to the stable manifold of $(0,0)$; then it will not intersect $B_0$ after the modification. Another topological reason is the following: if two points $x$ and $y$ belong to
$B_0$ and satisfy $f(x)=f(y)$, let $\alpha$ be a curve joining them within the basin and note that $A\alpha$ is an essential curve in the torus, obviously contained in $B_0$, but the basin does not contain
essential curves.\\
(3) Define $K$ as the complementary set of the union of the basins. Recall that the unstable manifolds of the Derived from Anosov are paralel straigth lines. The map $\pi$ carries lines into lines so there are no intersection between different (or the same) unstable manifolds. By a $\lambda$ lemma argument, this means that $K$ can also be seen as the closure of
the stable manifold of any of its saddles, for example, those in the boundary of the immediate basins. It follows that $K$ contains the stable manifold of every point: it is, consequently, a repeller basic piece. It has the same structure of the repeller basic piece of the torus, which is carried by $\pi$ to the sphere. \\
(4) That $K$ is indecomposable follows from Theorem \ref{indec}, because it has the local structure Cantor times interval. For each $n\geq 0$ each component of the preimage under $f^n$ of an immediate basin is a Wada Lake.\\
(5) Immediate from the fact that the  lift $\tilde f:\T^2 \to \T^2$ satisfies the growth rate inequality (see, for example Theorem 1.2 page 618 of \cite{bfgj}).\\
(6) As the basic pieces of a map are $C^1$ stable, it follows immediately that every $C^1$ perturbation of $f$ satisfies properties $(2)$ to $(5)$.
Obviously, to preserve the remaining properties it is necessary that the perturbation is smooth and a branched covering.

\section{Characterizing this family of examples}

Of course there is nothing particular about the matrix $A$ we used to construct the example in the previous section, and we could have used other non-expanding Anosov endomorphism.
However, the fact that the map $F:S^2 \to S^2$ we constructed in the previous section lifts to the covering $(\T^2, \pi)$ impose some restrictions on the kind of examples that can be
created this way.  We explore some of them in this final section in a self-contained way.\\

We say that $f:S^2\to S^2$ is a {\it Latt\`es map} if there exists a non-invertible covering map $\tilde f:\T^2 \to \T^2$ such that $\pi\tilde f = f\pi$ (the term Latt\'es map is
often used with a different meaning, but we adopt it here or convenience).

\begin{lemma}\label{ram} Let $f$ be a Latt\`es map.  Then, the critical values of $f$ are contained in the critical values of $\pi$: $f(Crit (f))\subset \pi (Crit (\pi))$.

\end{lemma}

\begin{proof} If $x\in S^2 \backslash \pi (Crit (\pi))$, then $\#\pi^{-1}(x)=2$ and $\#\tilde f ^{-1}(\pi^{-1}(x))=2d$, where $d>1$ is the degree of $\tilde f$.  So,
$\#\pi(\tilde f ^{-1}(\pi^{-1}(x)))= \#f^{-1}(x) \geq d$, and
therefore $x$ is not a critical value.

\end{proof}

\begin{lemma}\label{crit} Let $f$ be a Latt\`es map.  Then, the critical values of $\pi$ are not critical points of $f$: $Crit(f)\cap \pi (Crit (\pi))=\emptyset$.

\end{lemma}

\begin{proof} If $x\in Crit (\pi)$, then $\pi$ is $2:1$ at $x$. Therefore, if $\pi(x)$ is a critical point of $f$, $f\pi$ is $k:1$ at $x$, where $k>2$. But
$f\pi= \pi\tilde f$ that is at most $2:1$ at $x$.

\end{proof}

\begin{lemma}\label{inv} Let $f$ be a Latt\`es map.  Then, the critical values of $\pi$ are $f$-invariant: $f(\pi(Crit (\pi)))\subset \pi (Crit (\pi))$.

\end{lemma}

\begin{proof}  Let $x$ be a critical value of $\pi$. By the previous lemma $x$ is not a critical point for $f$,
then $f\pi$ is locally $2:1$ at $z=\pi^{-1}(x)$.  As $f\pi=\pi\tilde f$, $\tilde f(z)$ must be a critical point of $\pi$, that is $\pi\tilde f(z)=f\pi(z)=f(x)$ is a critical value of
$\pi$.

\end{proof}

As the critical values of $\pi$ is a set of four points, these points must be $f$-periodic. By Lemma \ref{ram}, the critical points of $f$ are pre-periodic. This already imposes a
restriction on our family of examples: they are all Thurston maps (orientation preserving branched coverings of the
sphere onto itself such that the postcritical set is finite). Furthermore, the following holds:

\begin{lemma} Let $f$ be a Latt\`es map.  Then, the critical points of $f$ are not periodic.

\end{lemma}

\begin{proof} This follows immmediately from lemmas \ref{crit} and \ref{inv}.

\end{proof}

As a consequence, Latt\`es maps cannot be topological polinomials, that is having a point $\infty \in S^2$ such that $f^{-1}(\infty)=\infty$. In particular, the question
as to whether there exists a topological polinomial supporting a completely invariant indecomposable continuum  remains open.\\

Regarding the growth rate inequality, one immediately obtains from Theorem 1.2 page 618 of \cite{bfgj}:

\begin{lemma} Let $f$ be a Latt\`es map such that the induced map on homology of its lift $\tilde f$ to the torus, $\tilde f_*:H_1(\T^2,\Z)\to H_1(\T^2,\Z)$ is hyperbolic.  Then,
$f$ satisfies the growth rate inequality.
\end{lemma}

Note that the hypothesis on the homology class of the lift is necessary, as the product of $z^2$ times an irrational rotation acting on $S^1\times S^1$ shows.\\

We finish by commenting that these examples are Thurston maps with signature $(2,2,2,2)$ which lie in the family of Thurston maps with non-hyperbolic orbifolds.  The growth
rate inequality problem for this larger family was studied by the authors in \cite{iprx2}. We showed that if $f$
is a  Thurston map with non hyperbolic orbifold, then either the growth rate inequality holds for $f$ or
$f$ has exactly two  critical points which are fixed and totally invariant.\\

{\bf Acknowledgments.} The authors are thankful to Professor Jan P. Boro\'nski for pointing us to Gammon's example and related work, and for several useful discussions over the years.\\

\end{document}